\documentclass{article}
\usepackage{amsfonts}
\usepackage{amsmath}

\newtheorem{theorem}{Theorem}

\newtheorem{definition}[theorem]{Definition}

\newtheorem{lemma}[theorem]{Lemma}

\newtheorem{proposition}[theorem]{Proposition}

\newenvironment{proof}[1][Proof]{\noindent\textbf{#1.} }{\ \rule{0.5em}{0.5em}}
\input{tcilatex}

\begin{document}

\title{A nonstandard proof for Szpilrajn's theorem}
\author{Abdelmadjid BOUDAOUD \\
Pure and Applied Mathematics Laboratory (LMPA),\\
Department of Mathematics,\\
Faculty of mathematics and Computer Sciences \\
University of M'sila, Algeria.\\
E-mail : boudaoudab@yahoo.fr }

\maketitle

\begin{abstract}
Recall that Szpilrajn (1930) (\cite{Egbert}, \cite{Karel}) states that on a given set, any partial order can be extended to a total order on the same set. In this work we give, in the context of the IST theory (\cite{Vanden},\cite{Francine},\cite{Nelson}),
a more constructive proof for this theorem. In addition, we benefit of the tools used to give some other results.

\end{abstract}



\section{Introduction}

This paper is placed in the framework of IST (\cite{Vanden},\cite{Francine},\cite{Nelson}), where
to prove the result announced above in the abstract, the main tools are the
use, instead of an infinite set $X$, of a finite subset $F=\left\{
x_{1},x_{2},...,x_{N}\right\} \subset X$ containing all standard elements of
$X$, transfer principle and standardization principle. In addition, since we
can construct several alternatives of linear order on a given set \cite{boudaoud} , we give some other applications, which are Propositions 6, 7
and 8. Let us recall the
\noindent \textbf{Standardization principle \ }: For all formula $F\left(
Z\right) $ internal or external, we have:

\begin{equation}
\forall ^{st}x\exists ^{st}y\forall ^{st}z\left[ z\in y\iff z\in x\text{ and
}F\left( z\right) \right]  \tag{S}
\end{equation}%
where $x$ is the referential set and $y$ is a standard set which is the
standardized of $\left\{ z\in x\text{ and }F\left( z\right) \right\} $. We
put $y=\left\{ z\in x\text{ and }F\left( z\right) \right\} ^{s}$. Thus, by
using the principle (S), we can associate a standard set for any given set.
Hence, it is a principle of construction.\newline
\begin{definition}
\noindent 1) If $T_{1}$ and $T_{2}$ are subsets of a linearly ordered set,
we call $T_{1}$ strictly dense (resp. dense) in $T_{2}$, if for every two
elements $a<b$ of $T_{2}$ there exists an element $c\in T_{1}$ with $a<c<b$
(resp. $a\leq c\leq b$).\newline

\noindent 2) Let $A$ and $B$ be a given sets. $A$ is called equivalent to $B$%
, written $A\sim B$, if there exists a function $f:A\longrightarrow B$ which
is one-one and onto.\newline
\end{definition}

\noindent \textbf{Notation.} Let $n$ be a positive integer, $n\approx
+\infty $ denotes that $n$ is unlimited.\newline

We need to recall that from the main result of \cite{boudaoud} we
immediately deduce
\begin{theorem}
 Let $X$ be any infinite standard set, let $%
G=\left\{ \beta _{1},\beta _{2},...,\beta _{N}\right\} \subset X$, $\left(
N\approx +\infty \right) $ be a subset containing all the standard elements
of $X$ and let $\Omega =\left\{ \left( \beta _{i},\beta _{j}\right) _{N\geq
i\geq 1\text{ },\text{ }N\geq j\geq i}\text{ }:\beta _{i},\beta _{j}\in
G\right\} \subset X\times X$. Then, the relation $\preceq $ defined by $%
\Omega ^{s}$ from $X$ to $X$ as follows%
\begin{equation*}
x\preceq y\text{ iff }\left( x,y\right) \in \Omega ^{s}
\end{equation*}%
is a standard total ordering on $X$.
\end{theorem}

\section{Main results}
\begin{theorem}(Szpilrajn $\left[ 1930\right] $ \textbf) \cite{Egbert} \textbf{.} Let $\left( P,\leq \right) $\ be a poset. Then
there exists a linear order $\leq _{\ast }$\ on $P$\ which contains $\leq $,
a so-called linear extension of $\leq $. In particular: if $a$\ and $b$\ are
incomparable elements, this linear order can be chosen in such a way that $%
a\leq _{\ast }b$\ holds.\newline
\end{theorem}
\begin{proof}
In order to prove this result we need the
following lemmas whose proofs are without use AC \cite{Egbert}.\newline

\begin{lemma}

$\left[ \cite{Egbert}, \text{ p. 53}\right]  \textit{.} $ \textit{Let }$%
\left( P_{1},R_{1}\right) $\textit{\ be a poset, }$a$\textit{\ and }$b$%
\textit{\ two incomparable elements of }$P_{1}$\textit{. Then there exists
an order relation }$\widetilde{R_{1}}$\textit{\ on }$P_{1}$\textit{\ which
contains }$R_{1}$\textit{\ and in which }$a\widetilde{R_{1}}b$\textit{\
holds.}
\end{lemma}
\begin{lemma}

$\left[ \cite{Egbert}, \text{, p. 27}\right] $ \textit{.} Every order $%
R_{2}$ on a finite set $P_{2}$ is a subset of a linear order on this set. In
other words: Every order relation on a finite set is extendible to a linear
order.
\end{lemma}

\noindent Hence, lemma 5 is a partial statement of the theorem of Szpilrajn.

Return to the proof in question. By transfer, we may assume that $P$, $\leq $%
, $a$ and $b$ are standard. We distinguish the two following cases.\newline

\noindent \textbf{A)} $P$ \textit{is finite.} By lemma 4, there exists an
order relation $\widetilde{\leq }$ on $P$ which contains $\leq $ and in
which $a\widetilde{\leq }b$ holds. Now, by lemma 5 \textit{there exists a
linear order }$\widetilde{\widetilde{\leq }}$ which contains $\widetilde{%
\leq }$. By taking $\widetilde{\widetilde{\leq }}$ for $\leq _{\ast }$, we
finish the proof for this case. Indeed, \textit{since }$\widetilde{%
\widetilde{\leq }}$ is \textit{a linear order, }$\leq _{\ast }$is also.
Since $\leq \subset \widetilde{\leq }\subset \widetilde{\widetilde{\leq }}%
=\leq _{\ast }$, $\leq _{\ast }$\textit{\ contains }$\leq $. Moreover, the
fact that $a\widetilde{\leq }b$ entails, since $\widetilde{\leq }\subset
\widetilde{\widetilde{\leq }}$, $a\widetilde{\widetilde{\leq }}b$ i.e. $%
a\leq _{\ast }b$.

\noindent \textbf{B)} $P$ \textit{is infinite. }Let $P_{1}=\left\{
y_{1},y_{2},...,y_{N}\right\} $ be a finite subset of $P$ containing all
standard elements of $P$, where $N$ is necessarily an unlimited integer. In
particular, $a$ and $b$ belong to $P_{1}$, since they are standard.
Now, let us consider the poset $\left( P_{1}\text{,}\leq \uparrow
P_{1}\right) $ and denote the order $\leq \uparrow P_{1}$ by $R_{1}$, where $%
\leq \uparrow P_{1}$ is the restriction of $\leq $ to $P_{1}$.

Then by lemma 4 there exists an order relation $\widetilde{R_{1}}$\ on $%
P_{1} $\ which contains $R_{1}$\ and in which $a\widetilde{R_{1}}b$\ holds.
Again, by lemma 5 there exists a linear order $R$ on $P_{1}$ which contains $%
\widetilde{R_{1}}$. Then $aRb$.

Since $\left( P_{1},R\right) $ is a chain, we can put $P_{1}=\left\{
x_{1},x_{2},...,x_{N}\right\} $ ($x_{i}\in \left\{
y_{1},y_{2},...,y_{N}\right\} $ for $1\leq i\leq N$), such that $x_{1}Rx_{2}$%
, $x_{2}Rx_{3}$, $x_{3}Rx_{4}$, $...$, $x_{N-1}Rx_{N}$. We notice here that $%
a$ (resp. $b$) corresponds to $x_{i_{1}}$ (resp. to $x_{i_{2}}$), where $i_{1}$%
, $i_{2}$ are in $\left\{ 1,2,...,N\right\} $ with $x_{i_{1}}Rx_{i_{2}}$ since $%
a\widetilde{R_{1}}b$. Let $G=\left\{ \left( x_{i},x_{j}\right) _{N\geq i\geq
1\text{ },\text{ }N\geq j\geq i}\mid x_{i},x_{j}\in P_{1}\right\} $ be the
graph of $\left( P_{1},R\right) $, where saying that $\left(
x_{i},x_{j}\right) \in G$ is the same to saying that $x_{i}Rx_{j}$. Let $%
G^{s}\subset P\times P$ be the standardization of $G$. It is known that $%
G^{s}$ defines a relation $\Gamma $ from $P$ to $P$ as follows%
\begin{equation*}
x\Gamma y\text{ iff }\left( x,y\right) \in G^{s}\text{.}
\end{equation*}

\noindent According to theorem 2, $\Gamma $ is a total order on $P$.

Now, we prove the remainder of the theorem. Let $(x,y)\in P^{2}$ be a
standard element verifying $x\leq y$. Then we have successively $xR_{1}y$, $x%
\widetilde{R_{1}}y$, $xRy$. Then $\left( x,y\right) \in G$ and therefore $%
\left( x,y\right) \in G^{s}$ because $(x,y)$ is standard. Hence $x\Gamma y$.
So

\begin{equation*}
\forall ^{st}\left( x,y\right) \in P^{2}\left[ \left( x\leq y\right)
\Longrightarrow x\Gamma y\right] \text{.}
\end{equation*}%
Then by transfer, $\forall \left( x,y\right) \in P^{2}\left[ \left( x\leq
y\right) \Longrightarrow x\Gamma y\right] $. Hence, $\Gamma $ contains%
\textit{\ }$\leq $. From what precedes, we have concerning $\left(
a,b\right) $: $aRb$. This shows that $\left( a,b\right) \in G$ i.e. $\left(
a,b\right) \in G^{s}$ since $\left( a,b\right) $ is standard. Consequently, $%
a\Gamma b$. Now, by taking $\Gamma $ for $\leq _{\ast }$ we finish the proof
for this case. By transfer, we conclude for all $P$.\newline
\end{proof}
\newline

\begin{proposition}
Let $X$ be any set. Suppose that $A$ and
$B$ are nonempty disjoint subsets of $X$. Then we can provide $X$ by a total
ordering $\preceq $ such that

\begin{equation*}
\forall x,y\in X\left[ \left( x\in A\text{ and }y\in B\right) \implies
x\preceq y\right] \text{.}
\end{equation*}%
\end{proposition}

\begin{proof}
By transfer, we suppose $X$, $A$ and $B$ are
standard. Let%
\begin{equation*}
F=\left\{ \beta _{1},\beta _{2},...,\beta _{N}\right\} \subset X
\end{equation*}%
be a finite subset containing all standard elements of $X$. Suppose that
this numbering of elements of $F$ has been made so that elements of $A\cap F$
are before those of $B\cap F$. Now, by theorem 2 and transfer we finish
the proof.
\end{proof}

\begin{proposition}
Let $X$ be any set. Let $S$ be a system
of mutually disjoint subsets of $X$. Then we can provide $X$ by a total
ordering such that

\begin{equation*}
\forall A,B\in S\left[ A<B\text{ or }B<A\right] \text{.}
\end{equation*}%
\end{proposition}

\begin{proof}
Assume, by transfer, that $X$ and $S$ are
standard. Let $F=\left\{ x_{1},x_{2},...,x_{\omega _{1}}\right\} $, $\omega
_{1}\cong +\infty $ (resp. $S=\left\{ A_{1},A_{2},...,A_{\omega
_{2}}\right\} $, $\omega _{2}\cong +\infty $ ) be a finite subset of $X$
(resp. of $S$) containing all standard elements of $X$ (resp. of $S$). For $%
j=1,2,...,\omega _{2}$, we put
\begin{equation*}
F\cap A_{j}=\left\{ y_{j,1},y_{j,2},...,y_{j,n_{j}}\right\}
\end{equation*}%
which is a finite set containing all standard elements of $A_{j}$. Put $%
L=F\setminus \underset{j=1}{\overset{\omega _{2}}{\cup }}\left( F\cap
A_{j}\right) =\left\{ y_{1},y_{2},...,y_{s}\right\} $.

Now, we rearrange elements of $F$ as follows : We begin by elements of $%
F\cap A_{1}$, after those of $F\cap A_{2}$, after those of $F\cap A_{3}$ and
so on. Thus $F$ becomes%
\begin{eqnarray*}
F &=&\left\{
y_{1,1},y_{1,2},...,y_{1,n_{1}},y_{2,1},y_{2,2},...,y_{2,n_{2}},...,y_{%
\omega _{2},1},y_{\omega _{2},2},...,y_{\omega _{2},n_{\omega _{2}}}\right\}
\\
&&\text{ \ \ \ \ \ \ \ \ \ \ \ \ \ \ \ \ \ \ \ \ \ \ \ \ \ \ \ \ \ \ \ \ \ \
\ \ \ \ \ \ \ \ \ \ \ \ \ \ \ \ \ \ \ \ \ \ \ \ \ \ \ \ \ \ }\cup \left\{
y_{1},y_{2},...,y_{s}\right\}
\end{eqnarray*}%
where $\overset{\omega _{2}}{\underset{j=1}{\sum }}n_{j}+s=\omega _{1}$.
Now, by theorem 2 and transfer we finish the proof.
\end{proof}

\begin{proposition}
Let $X$ and $Y$ be any sets. If $X\sim Y$
then we can provide $X\cup Y$ by a total ordering such that $X$ becomes
strictly dense in $Y$.
\end{proposition}

\begin{proof}
By transfer, we may assume that $X$ and $Y$ are
standard. Since $X\sim Y$ then there exists a mapping $\varphi
:Y\longrightarrow X$ which is one-one and onto. By transfer $\varphi $ is
standard. Let $F_{Y}=\left\{ y_{1},y_{2},...,y_{N}\right\} $ be a finite
subset of $Y$ containing all standard elements of $Y$. Then $\varphi \left[
F_{Y}\right] =\left\{ \varphi \left( y_{1}\right) ,\varphi \left(
y_{2}\right) ,...,\varphi \left( y_{N}\right) \right\} $ is a finite subset
of $X$ containing all standard elements of $X$. Let us put

\begin{center}
$F_{X\cup Y}=%
\begin{array}{c}
\left\{ y_{1},\varphi \left( y_{1}\right) ,y_{2},\varphi \left( y_{2}\right)
,y_{3},\varphi \left( y_{3}\right) ,y_{4},\varphi \left( y_{4}\right)
,\right. \\
\text{ \ \ \ \ \ \ \ \ \ \ \ \ \ \ \ \ \ \ \ \ \ \ \ \ }\left. ...,\varphi
\left( y_{N-1}\right) ,y_{N},\varphi \left( y_{N}\right) \right\}%
\end{array}%
$.
\end{center}

\noindent Then $F_{X\cup Y}$ is a finite subset of $X\cup Y$ containing all
the standard elements of $X\cup Y$. As well as in the theorem 2, we
construct from $F_{X\cup Y}$ the set $G$ and thereafter the set $G^{s}$
which is a standard linear order in $X\cup Y$.

Let $y_{i}$ and $y_{j}$ be standard elements of $Y$ such that $y_{i}\prec
y_{j}$ in $G^{s}$ which is equivalent to $i<j$. Then we have $y_{i}<\varphi
\left( y_{i}\right) <y_{j}$, where $\varphi \left( y_{i}\right) $ is a
standard element of $X$. Therefore%
\begin{equation*}
\forall ^{st}\left( \alpha ,\beta \right) \in Y^{2}\exists ^{st}\gamma \in X
\left[ \alpha <\beta \implies \alpha <\gamma <\beta \right] \text{.}
\end{equation*}%
Which, by transfer, entails $\forall \left( \alpha ,\beta \right) \in Y^{2}$
$\exists \gamma \in X$ $\left[ \alpha <\beta \implies \alpha <\gamma <\beta %
\right] $. Consequently $X$ is strictly dense in $Y$, where $X$ and $Y$ are
subsets of $X\cup Y$ which is linearly ordered by $G^{s}$.
\end{proof}


\end{document}